\documentclass[12pt,reqno,a4paper]{amsart}
\usepackage{setspace} 
\usepackage{amssymb}
\usepackage{amsmath}
\usepackage{amsthm}
\usepackage{lineno}
\makeatletter
\@namedef{subjclassname@2010}{%
  \textup{2010} Mathematics Subject Classification}
\makeatother
\numberwithin{equation}{section}

\newtheorem{theorem}{Theorem}[section]
\newtheorem{coro}{Corollary}[section]

\title[Ramanujan's type of product of theta functions and its applications]
{Ramanujan's type of product of theta functions and its applications}
\author[D. J. Prabhakaran]{D. J. Prabhakaran}
\address{Department of Mathematics \\
Anna University, MIT Campus\\
Chennai- 600025\\ India}
\email{asirprabha@gmail.com}
\author[N. Jayakumar]{N. Jayakumar}
\address{Department of Mathematics \\
Rajalakshmi Institute of Technology, Kuthambakkam\\
Chennai- 600124\\ India}
\email{jaaimaths@gmail.com}
\author[K.Ranjithkumar]{K.Ranjithkumar}
\address{Department of Mathematics \\
MIT Campus, Anna University\\
Chennai- 600044\\ India}
\email{ranjithkrkkumar@gmail.com}
\begin{document}
\begin{abstract}
In this paper, we initiate a generous amount of new-found general theorems for explicit evaluations of product of the theta functions $b_{m, n}$ using Kronecker's limit formula and other various novel explicit evaluations that were introduced thereupon. Also, we have obtained a few novel values of Ramanujan's class invariant $g_n$ using new-found values of $b_{m, n}$. Eventually, we have shown that $b_{m, n}$ is the unit in certain algebraic number fields.
\end{abstract}

\maketitle
\section{Introduction}
The illustrious mathematician Srinivasa Ramanujan defined a theta function
\begin{eqnarray*}
f(a, b)&:=&\sum^{\infty}_{n=-\infty}a^{\frac{n(n+1)} {2}} b^{\frac{n(n-1)} {2}}, \ \ \ |ab|<1.
\end{eqnarray*}
The most important special cases of $f(a,b)$ are as follows:
\begin{eqnarray}
\varphi(q) &:=& f(q,q)=\sum^{\infty}_{n=-\infty}q^{n^2}=(-q;q^2)^2_{\infty}(q^2;q^2)_{\infty}, \label{212} \\
\psi(q) &:=& f(q,q^3)=\sum^{\infty}_{n=0}q^{\frac{n(n+1)}{2}}=\frac{(q^2;q^2)_{\infty}}{(q;q^2)_{\infty}} \label{312},\\
f(-q)&:=&f(-q,-q^2)=\sum^{\infty}_{n=-\infty}(-1)^nq^{\frac{n(3n-1)}{2}}=(q;q)_{\infty},
\end{eqnarray}
where
\begin{eqnarray*}
(a;q)_{\infty}&:=&\prod^{\infty}_{n=1}\left(1-aq^{n-1}\right),\, |q|< 1
\end{eqnarray*}
In his first note book \cite{sr1} Ramanujan   defined a product of theta function,
\begin{equation} \label{eqamn}
 a_{m,n} =\frac{ne^{-\frac{\pi}{4}(n-1)\, \sqrt{\frac{m}{n}}}\,\psi^2\left(e^{-\pi\, \sqrt{mn}}\right)\, \varphi^2\left(-e^{-2\pi\sqrt{mn}}\right)}{\psi^2\left(e^{-\pi\,\sqrt{m/n}}\right)\varphi^2\left(-e^{-2\pi\sqrt{m/n}}\right)},
\end{equation}
where, $m$ and $n$ are positive integers, and he listed 18 particular values of $ a_{m,n}$.

All those 18 values were proved by B.C.Berndt et al. \cite{bel2} using the known Ramanajun's  class invariants $ G_n := 2^{-1/4}\, q^{-1/24}(-q;q^2)_{\infty},\, \, q=e^{-\pi\sqrt{n}}$ via modular equations.

Among them, 6 out of 18 values were proved using Kronecker's limit formula. Berndt and et al. \cite{bel2}, Naika \cite{3ms2,ms3,ms4,mss3,ms5}, Saika \cite{ns}, and Prabhakaran et al. \cite{DJPKRK1} established additional new-found values of $a_{m,n}$ using modular equations with known $G_{n}$ values.

Naika et al. \cite{3ms2} introduced new product of theta function $b_{m,n}$,
\begin{equation} \label{eqbmn}
 b_{m,n} =\frac{ne^{-\frac{\pi}{4}(n-1)\, \sqrt{\frac{m}{n}}}\,\psi^2\left(-e^{-\pi\, \sqrt{mn}}\right)\, \varphi^2\left(-e^{-2\pi\sqrt{mn}}\right)}{\psi^2\left(-e^{-\pi\,\sqrt{m/n}}\right)\varphi^2\left(-e^{-2\pi\sqrt{m/n}}\right)},
\end{equation}
Here, $m$ is any positive rational and $n$ is a positive integer. Also, they have found several explicit values for $b_{m,n}$ through modular equation with known $g_n = 2^{-1/4}\, q^{-1/24}(q;q^2)_{\infty},\,\, q=e^{-\pi\sqrt{n}}$ values.

In this paper we represent $b_{m,n}$ interms of Dedekind Eta function $ \eta(z) $, where,
\begin{equation} \label{2lem234}
\eta(z) = q^{\frac{1}{24}}\left(q; q\right)_{\infty} = q^{\frac{1}{24}}f(-q),       \ \ \  \ \ \ \ q=e^{2\pi iz}, \textrm{Im} z >0.
\end{equation}
And, we proffer novel explicit values for $b_{m,n}$ using Kronecker's limit formula. Besides, we verify that $b_{m,n}$ is a unit in particular imaginary quadratic field. From the new $b_{m,n}$ we have instituted a few now-found Ramanujan's class invariant $g_n$.\\
Now, we give certain relations with regard to $b_{m,n}$ with $g_n$, $\varphi(q),\, \eta(z)$
\begin{eqnarray}\label{mu}
 b_{m,n}=\frac{nq^{\frac{(n-1)}{4}}\psi^2\left(-q^n\right)\varphi^2\left(-q^{2n}\right)}{\psi^2\left(-q\right)\varphi^2\left(-q^2\right)},
\end{eqnarray}
where, $q=e^{-\pi\sqrt{\frac{m}{n}}}$.\\
Employing \eqref{212} and \eqref{312} in \eqref{mu}, we achieve
\begin{eqnarray}\label{mu1}
 b_{m,n}=\frac{nq^{\frac{(n-1)}{4}}(q^{2n};q^{2n})^2_{\infty}(-q;q^2)^2_{\infty}(q^{2n};q^{2n})^2_{\infty}(-q^2;q^2)^2_{\infty}}
 {(-q^{n};q^{2n})^2_{\infty}(q^{2};q^{2})^2_{\infty}(-q^{2n};q^{2n})^2_{\infty}(q^{2};q^{2})^2_{\infty}}.
\end{eqnarray}
Since $\displaystyle (q;q)_{\infty} = (q^2;q^2)_{\infty}(q;q^2)_{\infty}$.
\begin{eqnarray}\label{mu21}
\frac{(q^2;q^2)^2_{\infty}}{(-q;q^2)_{\infty}(-q^2;q^2)_{\infty}} &=& \frac{(q;q)^2_{\infty}}{(q;q^2)^2_{\infty}(-q;q^2)_{\infty}(-q^2;q^2)_{\infty}}.
\end{eqnarray}
Using Euler identity
\begin{eqnarray}\label{mu2}
\frac{(q^2;q^2)^2_{\infty}}{(-q;q^2)_{\infty}(-q^2;q^2)_{\infty}} &=& \frac{(q;q)^2_{\infty}}{(q;q^2)^2_{\infty}}(-q;q)_{\infty} \label{mu122}
\end{eqnarray}
Now, the equation (\ref{mu1}) becomes
\begin{eqnarray}\label{mu31}
 b_{m,n}=\frac{nq^{\frac{(n-1)}{4}}(q^{n};q^n)^4_{\infty}(-q;q)^2_{\infty}(q;q^2)^4_{\infty}}
 {(-q^{n};q^{n})^2_{\infty}(q^{n};q^{2n})^4_{\infty}(q;q)^4_{\infty}}.
\end{eqnarray}
Using Euler identity $(q;q^2)^2_{\infty}(-q;q)^2_{\infty}=1$, we get
\begin{eqnarray*}
 b_{m,n}&=&\frac{nq^{\frac{(n-1)}{4}}(q^{n};q^n)^4_{\infty}(-q;q)^4_{\infty}(q;q^2)^6_{\infty}}
 {(-q^{n};q^{n})^4_{\infty}(q^{n};q^{2n})^6_{\infty}(q;q)^4_{\infty}}\\
 &=& n\,\frac{\varphi^4(-q^n)}{\varphi^4(-q)}\, . \frac{g^6_{m/n}}{g^6_{mn}}
\end{eqnarray*}
Considering the left-hand side (\ref{mu21}) and deploying Euler identity twice, the equation (\ref{mu1}) can be interpreted as
\begin{eqnarray}\label{mu22}
 b_{m,n}&=&\frac{nq^{\frac{(n-1)}{4}}(q^{n};q^n)^4_{\infty}(q;q^2)^2_{\infty}}
 {(q;q^{2n})^2_{\infty}(q;q)^4_{\infty}}\\
 &=&\frac{n \, \eta^4\left(\frac{1}{2}\sqrt{-mn}\right)\, g^2_{m/n}}
 {\eta^4\left(\frac{1}{2}\sqrt{-m/n}\right)\, g_{mn}^2}\nonumber.
\end{eqnarray}
Further, equation (\ref{mu22}) can be restated as
\begin{eqnarray}
 b_{m,n}&=&\frac{nq^{\frac{(n-1)}{4}}(q^{n};q^n)^2_{\infty}(q^{2n};q^{2n})^2_{\infty}}
 {(q;q)^2_{\infty}(q^2;q^2)^2_{\infty}}\label{sadmu}\\
 &=&\frac{n\, \eta^2\left(\frac{1}{2}\sqrt{-mn}\right)\eta^2\left(\sqrt{-mn}\right)}
 {\eta^2\left(\frac{1}{2}\sqrt{-m/n}\right)\eta^2\left(\sqrt{-m/n}\right)}.\label{sadmu1}
\end{eqnarray}
K.G. Ramanathan \cite{kg1} pioneered in finding the explicit values of Ramanujan's class invariants $g_n,\, G_n,\,\lambda_n, \, \textrm{and }\, \, t_n$  using Kronecker's limit formula introduced by C.L.Seigel \cite{SCL11}, in his TIFR  lecture. He signified the product of theta function $a_{m,n}$ pertaining to Dedekind Eta function and proposed to determine the explicit values of $a_{m,n}$. Unfortunately, he died before he could establish what he proposed for. Later, B.C.Berndt et al. established  $g_n,\, G_n,\,\lambda_n,  \,\textrm{ and}\,  a_{m,n}$  using Kronecker's limit formula.

All those aforementioned endeavors gave us an impetus to protract the conception for $b_{m,n}$. Refer \cite{Berndt-notebook-5,bel1,bel2,beal,bee5,kg1,STHM} for preliminaries.

\noindent Let, $m>0$ be square free, and let, $K=\mathbb{Q}(\sqrt{-m})$, the imaginary quadratic field with discriminant $d$
\begin{equation*} \label{LEElem1}
d = \left\{
           \begin{array}{ll}
           \displaystyle  -4m, \ \ \ \ \ \ \ \  $if$ & \hbox{$-m \equiv 2, 3$ (mod 4);} \\
           \displaystyle  -m, \ \ $if$ & \hbox{$-m \equiv 1$ (mod 4).}
           \end{array}
         \right.
\end{equation*}
Let, $K=\mathbb{Q}\left(\sqrt{-m}\right)$, where, $m$ is a positive square-free integer. Set
\begin{equation} \label{EElem}
\Omega = \left\{
           \begin{array}{ll}
           \displaystyle  \sqrt{-m}, \ \ \ \ \ \ \ \  $if$ & \hbox{$-m \equiv 2, 3$ (mod 4);} \\
           \displaystyle  \frac{1+\sqrt{-m}}{2}, \ \ $if$ & \hbox{$-m \equiv 1$ (mod 4).}
           \end{array}
         \right.
\end{equation}
The ring of algebraic integers in $K$ is $\mathbb{Z} [\Omega]$. Therefore, each ideal class in the class group $C_k$ contains primitive ideals which are $Z$ modules of the form $B=\left[a, b+\Omega\right]$, where, $a$ and $b$ are rational integer $a>0$, and $a$ divides the norm of $(b+\Omega)$, $|b|\leq a/2,$ where, $a$ is the smallest positive integer in $B$ and $N(B)=a$. \\
\indent By the theorem (6.2) \cite{beal}, the group of classes of binary quadratic forms of discriminant $d$ is isomorphic to the narrow class group of the quadratic number field $K$. For $B=\left[a, b+\Omega\right],$ the corresponding primitive form is given by $Q(x, y) = a(x+y\omega)(x+y\overline{\omega}),$ where, $\omega=(b+\Omega)/a$. If $Q(x, y) = ax^2+2bxy+cy^2$ with discriminant $d$, then the corresponding primitive ideal is $\left[a, b+\Omega\right]$. First we will ascertain $b_{2m,n}$ is a unit in particular imaginary quadratic field.
\section{Explicit evaluation of $b_{2m,n}$.}
The following are the theorem's that are analogous to the theorem's 3.2 to 3.4 \cite{bel2}
\begin{theorem}
Let, $m$ and $n$  be the odd positive integers, such that, $2nm$ is a square free and $-2nm \equiv 2 \ (\textrm{mod} \ 4)$. Then, $b_{2m,n}$ is a unit.
\end{theorem}
\begin{proof}
By lemma 3.1 \cite{bel2} with  $\mathcal{I}_1 = \left[\Omega, 1 \right]$ and  $\mathcal{I}_2 = \left[\Omega, \ n \right]$ then $A=\left[
\begin{array}{cc}
   1 & 0 \\
   0 & n \\
\end{array}
\right]$ and $f=n$ with $\Omega = \sqrt{-2mn}$
\begin{eqnarray}\label{gA}
n^{-12}|g_A(\Omega)|^2 &=& n^{-12}\left(\frac{n^{-12}\Delta\binom{\Omega}{n}}{\Delta\binom{\Omega}{1}}\right)^2 = n^{-12}\left(\frac{n^{12}\Delta\left(\frac{\Omega}{n}\right)}{\Delta(\Omega)}\right)^2\nonumber \\
&=& n^{-12}\left(\frac{\Delta\left(\sqrt{\frac{-2m}{n}}\right)}{\Delta\left(\sqrt{\frac{-2mn}{1}}\right)}\right)^2
\end{eqnarray}
Similarly, $\mathcal{I}_3 = \left[\Omega, 2 \right]$ and $\mathcal{I}_4 = \left[\Omega, \ 2n \right]$, then,
$A=\left[
\begin{array}{cc}
1 & 0 \\
0 & n \\
\end{array}
\right]$ and $f=n$ with $\Omega = \sqrt{-2mn}$
\begin{eqnarray}\label{gA1}
n^{-12}\left|g_A\left(\frac{\Omega}{2}\right)\right|^2 &=& n^{-12}\left(\frac{\Delta\left(\sqrt{\frac{-m}{2n}}\right)}{\Delta\left(\sqrt{\frac{-mn}{2}}\right)}\right)^2
\end{eqnarray}
Equations (\ref{gA}) and (\ref{gA1}) are a unit. Which implies $b_{2m,n}$ is also a unit.\\
This enables us to precisely demonstrate $b_{2m,n}$ concerning the fundamental unit.
\end{proof}
\begin{theorem}\label{th2.1}
Let, $m$ and $n$ be the odd positive integers, so that, $2mn$ is square free and $-2mn \equiv 2 ( \textrm{mod} \ 4).$ Set, $-8mn = d_1d_2$ with $d_1 > 0$ and $d_i \equiv 0 \ \textrm{(or)}\ 1 (\textrm{mod} \ 4),$ $i = 1, 2$. Assuming $K=\mathbb{Q}\left(\sqrt{-2mn}\right)$ has the property that each genus contains exactly one ideal class, then,
\begin{eqnarray}
b^{-1}_{2m, n}&=&\prod \epsilon^{-\frac{wh_1h_2}{w_2}} \label{equ2.29}
\end{eqnarray}
where, $h, h_1,$ and $h_2$ are the class numbers of $K, \mathbb{Q}\left(\sqrt{d_1}\right)$, and $\mathbb{Q}\left(\sqrt{d_2}\right)$ respectively, $w$ and $w_2$ are the number of roots of unity in $K$ and $\mathbb{Q}\left(\sqrt{d_2}\right),$ respectively. And, $\epsilon$ is the fundamental unit in $\mathbb{Q}\left(\sqrt{d_1}\right)$ and the product is over all genus characters $\chi$ coalesced with the decomposition $-8mn = d_1d_2$, in such a way that,
$\chi\left(\left[ \Omega, \ n \right]\right)=\chi\left(\left[ \Omega, \ 2n \right]\right)=-1$ and therefore $d_1, d_2, h_1, h_2, w_2,$ and $\epsilon$ are dependent on $\chi$.
\end{theorem}
\begin{proof}
The proof is analogue to the proof of the Theorem 3.3 in \cite{bel2}, and so, we give only a few details.\\
Let, $B_0 = [\Omega, \ 1], B_1 = [\Omega, \ 2], B_2 = [\Omega, \ n]$ and $ B_3 = [\Omega, \ 2n]$. For $B= [b+\Omega, \ a]$ and $z = (b+\Omega)/a$. Set,
\begin{eqnarray}\label{gfjsk}
F(B) &=& \frac{|\eta(z)|^2}{\sqrt{a}}.
\end{eqnarray}
The number of genus characters, such that, $\chi(B_2)=\chi(B_3)=-1$ equals $h/4$. Considering the argument that led to the equation (3.3) \cite{bel1}, we find
\begin{eqnarray}
\prod \epsilon^{-\frac{wh_1h_2}{w_2}} &=& \left(\frac{F(B_2)F(B_3)}{F(B_0)F(B_1)}\right)^{h/4} \label{equ2.28}
\end{eqnarray}
by (\ref{sadmu1}) and \eqref{gfjsk}, we obtained
\begin{eqnarray}
\left(\frac{F(B_2)F(B_3)}{F(B_0)F(B_1)}\right)&=&b^{-1}_{m,n} \label{equ2.233}
\end{eqnarray}
From \eqref{equ2.28} and \eqref{equ2.233}, we complete the proof.
\end{proof}
There are only seven $b_{2m,n}$ values subsist for the discriminant of the quadratic field $1\leq d \leq 10,000$.
\begin{coro}
We have
\begin{eqnarray*}
b_{10, 3} &=& \left(\sqrt{2}-1\right)^2.
\end{eqnarray*}
\end{coro}
\begin{proof}
Let, $m, n = 5, 3,$ and set, $K=\mathbb{Q}\left(\sqrt{-120}\right)$, then, $\Omega = \sqrt{-30},$ $d=-120,$ $h=4,$ $w=2,$ and each genus encompasses one class. The unique character $\chi$ is in such a way that $\chi\left(\left[3, \ \Omega \right]\right)=\chi\left(\left[6, \ \Omega \right]\right)=-1$ corresponding to the decomposition $-120 = 8\times (-15)$. Also, $h_1=1, h_2=2, w_2=2,$ and $\epsilon =\sqrt{2}+1$. Thus, by $\eqref{equ2.29}$
\begin{eqnarray*}
b_{10, 3} &=& \left(\sqrt{2}+1\right)^{-2}=\left(\sqrt{2}-1\right)^2.
\end{eqnarray*}
\end{proof}
\begin{coro}  We have
\begin{eqnarray*}
b_{14, 3} &=& \left(\sqrt{3}-\sqrt{2}\right)^2.
\end{eqnarray*}
\end{coro}
\begin{proof}
Let, $m, n = 7, 3,$ and set $K=\mathbb{Q}\left(\sqrt{-168}\right)$, then, $\Omega = \sqrt{-42},$ $d=-168,$ $h=4,$ $w=2,$ and each genus includes one class. The unique character $\chi$ is in such a way, that $\chi\left(\left[3, \ \Omega \right]\right)=\chi\left(\left[6, \ \Omega \right]\right)=-1$ corresponding to the decomposition $-168 = 24\times (-7)$. Also, $h_1=1, h_2=1, w_2=2,$ and $\epsilon =5+2\sqrt{6}$. Thus, by $\eqref{equ2.29}$
\begin{eqnarray*}
b_{14, 3} &=& \left(5+2\sqrt{6}\right)^{-1}=\left(\sqrt{3}-\sqrt{2}\right)^2.
\end{eqnarray*}
\end{proof}
\begin{coro}  We have
\begin{eqnarray*}
b_{14, 5} &=& \left(\sqrt{10}-3\right)^2.
\end{eqnarray*}
\end{coro}
\begin{proof}
Let, $m, n = 7, 5,$ and set $K=\mathbb{Q}\left(\sqrt{-280}\right)$, then, $\Omega = \sqrt{-70},$ $d=-280,$ $h=4,$ $w=2,$ and each genus comprises one class. The unique character $\chi$ is in such a manner that $\chi\left(\left[5, \ \Omega \right]\right)=\chi\left(\left[10, \ \Omega \right]\right)=-1$ corresponding to the decomposition $-280 = 40\times (-7)$. Also, $h_1=2, h_2=1, w_2=2,$ and $\epsilon =\sqrt{10}+3$. Thus, by $\eqref{equ2.29}$
\begin{eqnarray*}
b_{14, 5} &=& \left(\sqrt{10}+3\right)^{-2}=\left(\sqrt{10}-3\right)^2.
\end{eqnarray*}
\end{proof}
\begin{coro}  We have
\begin{eqnarray*}
b_{26, 3} &=& \left(\sqrt{2}-1\right)^4.
\end{eqnarray*}
\end{coro}
\begin{proof}
Let, $m, n = 13, 3,$ and set $K=\mathbb{Q}\left(\sqrt{-312}\right)$, then, $\Omega = \sqrt{-78},$ $d=-312,$ $h=4,$ $w=2,$ and each genus has one class. The unique character $\chi$ in such a manner that $\chi\left(\left[3, \ \Omega \right]\right)=\chi\left(\left[6, \ \Omega \right]\right)=-1$ corresponding to the decomposition $-312 = 8\times (-39) $. Also, $h_1=1, h_2=4, w_2=2,$ and $\epsilon =\sqrt{2}+1$. Thus, by $\eqref{equ2.29}$
\begin{eqnarray*}
b_{26, 3} &=& \left(\sqrt{2}+1\right)^{-4}=\left(\sqrt{2}-1\right)^4.
\end{eqnarray*}
\end{proof}
\begin{coro}  We have
\begin{eqnarray*}
b_{34, 3} &=& \left(\sqrt{17}-4\right)^2.
\end{eqnarray*}
\end{coro}
\begin{proof}
Let, $m, n = 17, 3,$ and set $K=\mathbb{Q}\left(\sqrt{-408}\right)$, then, $\Omega = \sqrt{-102},$ $d=-408,$ $h=4,$ $w=2,$ and each genus  consists of one class. The unique character $\chi$ is in such a way that $\chi\left(\left[3, \ \Omega \right]\right)=\chi\left(\left[6, \ \Omega \right]\right)=-1$ corresponding to the decomposition $-408 = 17\times (-24)$. Also, $h_1=1, h_2=2, w_2=2,$ and $\epsilon =\sqrt{17}+4$. Thus, by $\eqref{equ2.29}$
\begin{eqnarray*}
b_{34, 3} &=& \left(\sqrt{17}+4\right)^{-2}=\left(\sqrt{17}-4\right)^2.
\end{eqnarray*}
\end{proof}
\begin{coro}  We have
\begin{eqnarray*}
b_{26, 5} &=& \left(\sqrt{65}-8\right)^2.
\end{eqnarray*}
\end{coro}
\begin{proof}
Let $m, n = 13, 5,$ and set $K=\mathbb{Q}\left(\sqrt{-520}\right)$, then, $\Omega = \sqrt{-130},$ $d=-520,$ $h=4,$ $w=2,$ and each genus contains one class. The unique character $\chi$ is such that $\chi\left(\left[5, \ \Omega \right]\right)=\chi\left(\left[10, \ \Omega \right]\right)=-1$ corresponding to the decomposition $-520 = 65\times (-8) $. Also, $h_1=2, h_2=1, w_2=2,$ and $\epsilon =\sqrt{65}+8$. Thus, by $\eqref{equ2.29}$
\begin{eqnarray*}
b_{26, 5} &=& \left(\sqrt{65}+8\right)^{-2}=\left(\sqrt{65}-8\right)^2.
\end{eqnarray*}
\end{proof}
\begin{coro}  We have
\begin{eqnarray*}
b_{38, 5} &=& \left(\sqrt{2}-1\right)^8.
\end{eqnarray*}
\end{coro}
\begin{proof}
Let $m, n = 19, 5,$ and set $K=\mathbb{Q}\left(\sqrt{-760}\right)$, then, $\Omega = \sqrt{-190},$ $d=-760,$ $h=4,$ $w=2,$ and each genus contains one class. The unique character $\chi$ is in such a way that $\chi\left(\left[5, \ \Omega \right]\right)=\chi\left(\left[10, \ \Omega \right]\right)=-1$ corresponding to the decomposition $-760 = 8\times (-95) $. Also, $h_1=1, h_2=8, w_2=2,$ and $\epsilon =\sqrt{2}+1$. Thus, by $\eqref{equ2.29}$
\begin{eqnarray*}
b_{38, 5} &=& \left(\sqrt{2}+1\right)^{-8}=\left(\sqrt{2}-1\right)^8.
\end{eqnarray*}
\end{proof}
The aforementioned values were found previously using modular equations with known $g_n$ values by Naika \cite{3ms2}.
\begin{theorem}\label{th876}
Let, $m$ and $n$ be the odd positive integers, so that, $2mn$ is square free and $-2mn \equiv 2 ( \textrm{mod} \ 4).$ Set, $-8mn = d_1d_2$ with $d_1 > 0$ and $d_i \equiv 0 \ \textrm{(or)} 1 (\textrm{mod} \ 4),$ $i = 1, 2$. Presuming that $K=\mathbb{Q}\left(\sqrt{-2mn}\right)$ has the property that each genus contains exactly two ideal classes, then,
\begin{eqnarray}
b^{2}_{2m, n}&=&\prod \epsilon^{-\frac{wh_1h_2}{w_2}} \label{equkh6345}
\end{eqnarray}
where, $h, h_1,$ and $h_2$ are the class number of $K, \mathbb{Q}\left(\sqrt{d_1}\right),$ and $\mathbb{Q}\left(\sqrt{d_2}\right)$ respectively. Then, $w$ and $w_2$ are the number of roots of unity in $K$ and $\mathbb{Q}\left(\sqrt{d_2}\right),$ respectively. And, $\epsilon$ is the fundamental unit in $\mathbb{Q}\left(\sqrt{d_1}\right)$ and the product is the overall genus characters $\chi$ which is  associated with the decomposition $-8mn = d_1d_2$, such that,
$\chi\left(\left[ \Omega, \ n \right]\right)=-1$, and consequently, $d_1, d_2, h_1, h_2, w_2,$ and $\epsilon$ are dependent on $\chi$.
\end{theorem}
\begin{proof}
Since  $\left[ \Omega, \ 2 \right]$ is in the principal genus and $\left[ \Omega, \ n \right]$ and $\left[ \Omega, \ 2n \right]$ are in the same(non-principal) genus. Let, $B_0 = [\Omega, \ 1], B_1 = [\Omega, \ 2], B_2 = [\Omega, \ n],$ and $ B_3 = [\Omega, \ 2n]$. Let $B = [b+\Omega, \ a]$ and $z = (b+\Omega)/a$, let, $F(B)$ be defined by \eqref{gfjsk}. Then, using identically similar arguments as we used in the proof of the theorem \ref{th2.1}. We conclude that
\begin{eqnarray*}
\prod \epsilon^{-\frac{wh_1h_2}{w_2}} &=& \left(\frac{F(B_2)F(B_3)}{F(B_0)F(B_1)}\right)^\frac{h}{4}
\end{eqnarray*}
 Nevertheless, here, $h=8$ and using (\ref{sadmu1}) , then, the aforementioned equation turns out to be
 \begin{eqnarray}
 \prod \epsilon^{-\frac{wh_1h_2}{w_2}} &=& \left(\frac{F(B_2)F(B_3)}{F(B_0)F(B_1)}\right)^2=b_{2m,n}^2\label{equkh634}.
\end{eqnarray}
We complete the proof.
\end{proof}
There are only three $b_{2m,n}$ values exists for the discriminant of the quadratic field $1\leq d \leq 10,000$. The following are new-found values of $b_{2m,n}$ which will be used in finding the novel values of Ramanujan's class invariant $g_n$ in the next section.
\begin{coro} We have
\begin{eqnarray*}
b_{34, 7} &=& \left(3\sqrt{2}-\sqrt{17}\right)^2\left(\sqrt{17}-4\right)^2.
\end{eqnarray*}
\end{coro}
\begin{proof}
Let, $m, n = 17, 7$, and set, $K=\mathbb{Q}\left(\sqrt{-952}\right)$, then, $\Omega = \sqrt{-238},$ $d=-952,$ $h=8,$ $w=2,$ and each genus comprehends two classes. The unique character $\chi$ is in such a manner that $\chi\left(\left[7, \ \Omega \right]\right)=-1$ corresponding to the decomposition $-952 = 136\times (-7) = 17\times (-56)$. Also, $h_1=2, h_2=1, w_2=2, $ and $\epsilon =35+6\sqrt{34}$; $h_1=1, h_2=4, w_2=2, $ and $\displaystyle \epsilon =\sqrt{17}+4$ respectively. Thus, by $\eqref{equkh6345}$
\begin{eqnarray*}
b^2_{34, 7} &=& \left(35+6\sqrt{34}\right)^{-2}\left(\sqrt{17}+4\right)^{-4}\\
b_{34, 7} &=& \left(3\sqrt{2}-\sqrt{17}\right)^{2}\left(\sqrt{17}-4\right)^{2}.
\end{eqnarray*}
\end{proof}
\begin{coro}\label{cor0} We have
\begin{eqnarray*}
b_{46, 7} &=& \left(16\sqrt{23}-29\sqrt{7}\right)\left(58\sqrt{2}-31\sqrt{7}\right).
\end{eqnarray*}
\end{coro}
\begin{proof}
Let, $m, n = 23, 7,$ and set $K=\mathbb{Q}\left(\sqrt{-1288}\right)$, then, $\Omega = \sqrt{-322},$ $d=-1288,$ $h=8,$ $w=2,$ and each genus includes two classes. The unique character $\chi$ is in such a way that $\chi\left(\left[7, \ \Omega \right]\right)=-1$ corresponding to the decomposition $-1288 = 161\times (-8) = 184\times (-7)$. Also, $h_1=1, h_2=1, w_2=2, $ and $\epsilon =11775+928\sqrt{161}$; $h_1=1, h_2=3, w_2=2, $ and $\displaystyle \epsilon =15+4\sqrt{14}$ respectively. Thus, by $\eqref{equkh6345}$
\begin{eqnarray*}
b^2_{46, 7} &=& \left(11775+928\sqrt{34}\right)^{-1}\left(15+4\sqrt{14}\right)^{-3}\\
b_{46, 7} &=& \left(16\sqrt{23}-29\sqrt{7}\right)\left(58\sqrt{2}-31\sqrt{7}\right).
\end{eqnarray*}
\end{proof}
\begin{coro}\label{cor1} We have
\begin{eqnarray*}
b_{94, 7} &=& \left(412\sqrt{7}-159\sqrt{47}\right)\left(732\sqrt{2}-151\sqrt{47}\right).
\end{eqnarray*}
\end{coro}
\begin{proof}
Let, $m, n = 47, 7,$ and set $K=\mathbb{Q}\left(\sqrt{-2632}\right)$, then, $\Omega = \sqrt{-658},$ $d=-2632,$ $h=8,$ $w=2,$ and each genus comprises two classes. The unique character $\chi$ is such that $\chi\left(\left[7, \ \Omega \right]\right)=-1$ corresponding to the decomposition $-2632 = 329\times (-8) = 376\times (-7)$. Also, $h_1=1, h_2=1, w_2=2, $ and $\epsilon =2376415+131016\sqrt{329}$; $h_1=1, h_2=1, w_2=2, $ and $\displaystyle \epsilon =2143295+221064\sqrt{94}$ respectively. Thus, by $\eqref{equkh6345}$
\begin{eqnarray*}
b^2_{94, 7} &=& \left(2376415+131016\sqrt{329}\right)^{-1}\left(2143295+221064\sqrt{94}\right)^{-1}\\
b_{94, 7} &=& \left(412\sqrt{7}-159\sqrt{47}\right)\left(732\sqrt{2}-151\sqrt{47}\right).
\end{eqnarray*}
\end{proof}
\begin{theorem}\label{thmgn}Modular equation of degree 7 \cite{Berndt-notebook-3}p.315
\begin{eqnarray*}
2\sqrt{2}\left( \left( g_{7n} g_\frac{n}{7}\right)^{3}+\frac{1}{\left( g_{7n} g_\frac{n}{7}\right)^{3}}\right)
 &=&
\left(\frac{g_{7n}}{g_\frac{n}{7}} \right)^{4}+\left(\frac{g_\frac{n}{7}}{g_{7n}}\right)^{4}-7
\end{eqnarray*}
\end{theorem}
\section{Applications of $b_{2m,n}$}
In this section, we arrive at two of values of Ramanujan's Class invariants $g_n$ using the new-found $b_{2m,n}$ which is specified in earlier section.
\begin{theorem} We have
\begin{eqnarray*}
g_{322}&=&
\left( \sqrt{\frac{31+4\sqrt{46}}{4}}-\sqrt{\frac{27+4\sqrt{46}}{4}}\right)^{\frac{1}{2}}
\left( \sqrt{\frac{55+8\sqrt{46}}{4}}-\sqrt{\frac{51+8\sqrt{46}}{4}}\right)^{\frac{1}{2}}
\end{eqnarray*}
\end{theorem}
\begin{proof}
From corollary (\ref{cor0}) , we have
\begin{eqnarray*}
b_{46, 7} &=& \left(16\sqrt{23}-29\sqrt{7}\right)\left(58\sqrt{2}-31\sqrt{7}\right)
\end{eqnarray*}
In theorem $2.7$ \cite{3ms2}, let\ $ m = 46$ and \ $ x = \left(\dfrac{g_{322}}{g_\frac{46}{7}}\right)^2 + \left(\dfrac{g_\frac{46}{7}}{g_{322}}\right)^2 $, we have
\begin{eqnarray*}
b_{46,7}+\dfrac{1}{b_{46,7}} &=&\dfrac{1}{7}\left( x^{3}-11x\right)
\end{eqnarray*}
And, by substituting \ $ b_{46,7}$ value and solving the aforementioned equation by selecting the suitable root, we obtain
\begin{eqnarray*}
\dfrac{g_{322}}{g_{\frac{46}{7}}} &=&\sqrt{\frac{29+4\sqrt{46}}{2}-\dfrac{1}{2}\sqrt{1573+232\sqrt{46}}}
\end{eqnarray*}
After manipulation, we get
\begin{eqnarray}\label{g322eq1}
\dfrac{g_{322}}{g_{\frac{46}{7}}} &=&\sqrt{\frac{31+4\sqrt{46}}{4}-\sqrt \frac{27+4\sqrt{46}}{4}}
\end{eqnarray}
From the theorem (\ref{thmgn}), we have
\begin{eqnarray*}
\left( g_{322} \ g_{\frac{46}{7}}\right)^{3}+\left( g_{322} \ g_{\frac{46}{7}}\right)^{-3} &=&4\sqrt{38551+5684\sqrt{46}}
\end{eqnarray*}
On solving the equation and choosing the apt root, we get
\begin{eqnarray*}
 g_{322} \ g_{\frac{46}{7}}&=&\left(308407+45472 \sqrt{46}-28\sqrt{242639293+35775212\sqrt{46}} \right)^\frac{1}{6}
\end{eqnarray*}
After manipulation, we have
\begin{eqnarray}\label{g322eq2}
 g_{322} \ g_{\frac{46}{7}}&=&\sqrt{\frac{55+8\sqrt{46}}{4}}-\sqrt{\frac{51+8\sqrt{46}}{4}}
\end{eqnarray}
From (\ref{g322eq1}) and (\ref{g322eq2}), we obtained
\end{proof}
\begin{theorem} We have
\begin{eqnarray*}
g_{658}&=&
\left( \sqrt{\frac{161+42\sqrt{14}}{4}}-\sqrt{\frac{157+42\sqrt{14}}{4}}\right)^{\frac{1}{2}}
\left( \sqrt{135+36\sqrt{14}}-\sqrt{134+36\sqrt{14}}\right)^{\frac{1}{2}}
\end{eqnarray*}
\end{theorem}
\begin{proof}
From corollary (\ref{cor1}), we have
\begin{eqnarray*}
b_{94, 7} &=& \left(412\sqrt{7}-159\sqrt{47}\right)\left(732\sqrt{2}-151\sqrt{47}\right).
\end{eqnarray*}
In theorem $2.7$ \cite{3ms2}, let\ $ m = 94 $ and \  $ x = \left(\dfrac{g_{658}}{g_\frac{94}{7}}\right)^2 + \left(\dfrac{g_\frac{94}{7}}{g_{658}}\right)^2 $, we have
\begin{eqnarray*}
b_{94,7}+\dfrac{1}{b_{94,7}} &=&\dfrac{1}{7}\left( x^{3}-11x\right)
\end{eqnarray*}
After substituting the value \ $ b_{94,7}$ and solving the aforementioned equation, choose the appropriate root, then, we get
\begin{eqnarray*}
\dfrac{g_{658}}{g_{\frac{94}{7}}} &=&\sqrt{\frac{159}{2}+21\sqrt{14}-\dfrac{1}{2}\sqrt{7\left( 7139+1908\sqrt{14}\right) }}
\end{eqnarray*}
After manipulation, we obtain
\begin{eqnarray}\label{g658eq1}
\dfrac{g_{658}}{g_{\frac{94}{7}}} &=&\sqrt{\frac{161+42\sqrt{14}}{4}}-\sqrt{ \frac{157+42\sqrt{14}}{4}}
\end{eqnarray}
From the theorem (\ref{thmgn}), we have
\begin{eqnarray*}
\left( g_{658} \ g_{\frac{94}{7}}\right)^{3}+\left( g_{658} \ g_{\frac{94}{7}}\right)^{-3} &=&18\sqrt{694\sqrt{2}+371\sqrt{7}}
\end{eqnarray*}
By solving the equation and choosing the appropriate root, we have
\begin{eqnarray*}
 g_{658} \ g_{\frac{94}{7}}&=&\left(312134957+83421576\sqrt{14}-18\sqrt{14\left( 42957773847867+11480947988374\sqrt{14}\right) } \right)^\frac{1}{6}
\end{eqnarray*}
After manipulation, we have
\begin{eqnarray}\label{g658eq2}
 g_{658} \ g_{\frac{94}{7}}&=&\sqrt{135+36\sqrt{14}}-\sqrt{134+36\sqrt{14}}
\end{eqnarray}
From (\ref{g658eq1}) and (\ref{g658eq2}), we obtained.
\end{proof}

\end{document}